\documentclass{amsart}
\usepackage{graphicx}
\usepackage{amsmath}
\usepackage{amscd}
\usepackage{amsfonts}
\usepackage{amssymb}

\numberwithin{equation}{section}

\newtheorem{theorem}{Theorem}[section]
\newtheorem{lemma}[theorem]{Lemma}
\newtheorem{proposition}[theorem]{Proposition}

\newtheorem{definition}[theorem]{Definition}

\newtheorem{example}[theorem]{Example}
\newtheorem{remark}[theorem]{Remark}

\newcommand{\be}{\begin{equation}}
\newcommand{\ee}{\end{equation}}
\newcommand{\bes}{\begin{equation*}}
\newcommand{\ees}{\end{equation*}}

\newcommand{\cL}{\mathcal{L}}

\newcommand{\cA}{\mathcal{A}}

\newcommand{\lel}{\left\langle}
\newcommand{\rir}{\right\rangle}

\newcommand{\mb}[1]{\mathbb{#1}}

\newcommand{\arv}{H^2_d \otimes \mb{C}^r}

\begin{document}

\title[Stable division]{Stable polynomial division and essential normality of graded Hilbert modules}

\author{Orr Moshe Shalit}
\address{Department of Pure Mathematics, University of Waterloo}
\email{oshalit@math.uwaterloo.ca}





\begin{abstract}
The purpose of this paper is to initiate a new attack on Arveson's resistant conjecture, that all graded submodules of the $d$-shift Hilbert module $H^2$ are essentially normal. We introduce the \emph{stable division property} for modules (and ideals): a normed module $M$ over the ring of polynomials in $d$ variables has the stable division property if it has a generating set $\{f_1, \ldots, f_k\}$ such that every $h \in M$ can be written as $h = \sum_i a_i f_i$ for some polynomials $a_i$ such that $\sum \|a_i f_i\| \leq C\|h\|$. We show that certain classes of modules have this property, and that the stable decomposition $h = \sum a_i f_i$ may be obtained by carefully applying techniques from computational algebra. We show that when the algebra of polynomials in $d$ variables is given the natural $\ell^1$ norm, then every ideal is linearly equivalent to an ideal that has the stable division property. We then show that a module $M$ that has the stable division property (with respect to the appropriate norm) is $p$-essentially normal for $p > \dim(M)$, as conjectured by Douglas. This result is used to give a new, unified proof that certain classes of graded submodules are essentially normal. Finally, we reduce the problem of determining whether all graded submodules of the $d$-shift Hilbert module are essentially normal, to the problem of determining whether all ideals generated by quadratic scalar valued polynomials are essentially normal.
\end{abstract}

\maketitle

\section{Introduction}

\subsection{The basic setup}
Let $\cA_d := \mb{C}[z_1, \ldots, z_d]$ be the algebra of complex polynomials in $d$ variables. We use the usual multi-index notation: if $\alpha = (\alpha_1, \ldots, \alpha_d) \in \mb{N}^d$ is a multi-index, then $|\alpha| = \alpha_1 + \ldots + \alpha_d$ and
\bes
z^\alpha = z_1^{\alpha_1} z_2^{\alpha_2} \cdots z_d^{\alpha_d}.
\ees
 
We denote by $\cA_d \otimes \mb{C}^r$ the finite multiplicity versions of $\cA_d$.  In this note we are interested in the case where there is some norm, always denoted  $\|\cdot\|$, on $\cA_d$. We will consider in detail two norms.

\emph{The natural $\ell_1$ norm:} For $p(z) = \sum_\alpha c_\alpha z^\alpha$ we define
\be\label{eq:l1}
\|p\| = \sum_\alpha |c_\alpha|.
\ee 

\emph{The $H^2$ norm:} 
We give $\cA_d$ an inner product by declaring that all monomials are orthogonal one to the other, and a monomial has norm
\be\label{eq:H2}
\|z^\alpha\|^2 = \frac{\alpha_1! \cdots \alpha_d!}{|\alpha|!}.
\ee
$H^2 = H^2_d$ will denote the Hilbert space obtained by completing $\cA_d = \mb{C}[z_1, \ldots, z_d]$ with respect to the above mentioned inner product. This space is also known as ``Symmetric Fock Space", or ``the Drury-Arveson" space.

$\cA_d$ has a natural grading that extends naturally to its finite multiplicity versions and to its completions with respect to the various norms.
We write the grading of 
$H^2_d$ as $H_0 + H_1 + H_2 + \ldots$. Thus, $H_k$ will also stand for the space of homogeneous polynomials of degree $k$.

A \emph{homogeneous ideal} (resp., \emph{module}) is an ideal (resp., module) generated by homogeneous polynomials. We say that $M$ is a \emph{graded submodule} of $\arv$ if it is the completion of a homogeneous module. Whenever $M \subseteq H^2_d \otimes \mb{C}^r$ is a graded submodule of $H^2_d \otimes \mb{C}^r$, we write the grading of $M$ as $M = M_0 + M_1 + M_2 + \ldots$.

\subsection{Stable polynomial division}

Let $M$ be a submodule of $\cA_d \otimes \mb{C}^r$ and let $\{f_1, \ldots, f_k\}$ be a generating set. Then every $h \in M$ can be written as a combination
\be\label{eq:h}
h = a_1 f_1 + \ldots + a_k f_k ,
\ee
with $a_i \in \cA_d$, $i=1, \ldots, k$. A natural question that arises is whether this can be done in such a way that that the terms $a_i f_i$ are controlled by the size of $h$. That is, does there exist a constant $C$ such that 
\be\label{eq:stablediv2}
\sum \|a_i f_i\|^2 \leq C \|h\|^2 
\ee
for all $h \in M$.


\begin{definition}
Let $M$ be a submodule of $\cA_d \otimes \mb{C}^r$. We say that $M$ has the \emph{stable division property} if there is a set $\{f_1, \ldots, f_k\} \subset  M$ that generates $M$ as a module, and there exists a constant $C$, such that for any polynomial $h \in M$ one can find $a_1, \ldots, a_k \in \cA_d$ such that
(\ref{eq:h})  and (\ref{eq:stablediv2}) hold. In this case, we also say that $M$ has stable division constant $C$. The set $\{f_1, \ldots, f_k\}$ is said to be a \emph{stable generating set} for $M$.
\end{definition}

\begin{remark}\emph{
A generating set for a module with the stable division property is not necessarily a stable generating set (see Example \ref{expl:non_stable_basis}).}
\end{remark}

\begin{remark}
\emph{When $M$ is a graded module it suffices to check (\ref{eq:h}) and (\ref{eq:stablediv2}) for $h$ homogeneous.}
\end{remark}

\begin{remark}\emph{
Note that condition (\ref{eq:stablediv2}) is equivalent to }
\be\label{eq:stablediv1}
\sum \|a_i f_i\| \leq C' \|h\| ,
\ee
\emph{when the finite set of generators is held fixed. }
\end{remark}

For an example of a module with the above property, note that any principal submodule of $\cA_d \otimes \mb{C}^r$ has the stable division property. On the other hand, we do not know whether or not there are submodules of $\cA_d \otimes \mb{C}^r$ that do not enjoy this property. Of greatest interest to our purposes is the case where $M$ is generated by \emph{homogeneous} polynomials, and we shall focus mainly on this case.

Although the literature contains some recent treatment of numerical issues arising in computational algebra (see, e.g., \cite{AFT,KSW,MT}) and although questions of effective computation in algebraic geometry have been considered for some time (see, e.g., this survey \cite{BM}), it does not seem that the problems with which we deal here have been addressed.

Below we will give some additional examples of modules with the stable division property. But before that, let us indicate some difficulties that arise in this context.

\begin{example}\label{ex:x^2+xy}\emph{
In the following discussion we will use some standard terminology from computational algebraic geometry (see the appendix for a review). Consider the ideal $I \subset \mb{C}[x,y]$ generated by the set $B = \{x^2 + 2xy, y^2\}$. One can check that $B$ is a Groebner basis for $I$. There is a  standard and well known algorithm that, given $h \in I$, finds coefficients $a_1, a_2 \in \cA_d$ such that $h = a_1 f_1 + a_2 f_2$ \cite[p. 63]{CLO92}. However, this division algorithm is not stable. For example, running the division algorithm on $x^{n+2}$ gives the output
\bes
x^{n+2} = \big[x^{n} - 2x^{n-1}y + 4x^{n-2}y^2 + \ldots + (-2)^{n}y^{n}\big](x^2 + 2xy) + \big[(-2)^{n+1}xy^{n-1} \big] y^2.
\ees
Thus, while the polynomials $x^{n+2}$ have norm $1$, running the division algorithm naively exhibits these polynomials as the sum of two terms of norm $\sim 2^n$. In particular, the division algorithm may be numerically unstable.}
\end{example}

Note that one may also write
\bes
x^{n+2} = \big[x^{n} - 2x^{n-1}y \big](x^2 + 2xy) + \big[4x^n \big] y^2. 
\ees
We will show below that in the two variable case, a slight modification of the above mentioned algorithm will always give the desired result. However, it is not clear whether it is possible to design an algorithm that will make the correct choices to produce optimal coefficients in the general $d$-variable case. In Section \ref{sec:H^2} and \ref{sec:ell1} we treat specific classes of modules for which we can show that the stable division property holds. We will show that with respect to the $H^2$ norm ideals generated by linear polynomials, arbitrary ideals in $\mb{C}[x,y]$, finite dimensional ideals, as well as modules generated by monomials, have the stable division property. These classes of modules can be seen (using the same proofs) to have the stable division property with respect to the $\ell^1$ norm too, but with respect to the $\ell^1$ norm we in fact show that every ideal is linearly equivalent to an ideal that has the stable division property.

\subsection{The $d$-shift and essential normality}
We now explain the reason that brought us to study stable division.

On $\arv$ we may define natural multiplication operators $Z_1, \ldots, Z_d$ as follows:
\bes
Z_i f(z) = z_i f(z) \,\, , \,\, f \in \arv.
\ees
The $d$-tuple $(Z_1, \ldots, Z_d)$ is known as \emph{the $d$-shift}, and has been studied extensively in \cite{Arv98} and since. Arveson showed that the commutators $[Z_i,Z_j^*]$ belong to the Schatten class $\cL^p$ for all $p>d$, thus, in particular, they are compact. This is significant - see \cite{Arv98} for ramifications.

 Given a graded submodule $M \subseteq \arv$, one may obtain two other $d$-tuples by compressing $(Z_1, \ldots, Z_d)$ to $M$ and to $M^\perp$:
 \bes
 (A_1, \ldots, A_d) = \left(Z_1\big|_M, \ldots, Z_d\big|_M\right) \,,
 \ees
 and
\bes
 (B_1, \ldots, B_d) = \left(P_{M^\perp}Z_1\big|_{M^\perp}, \ldots, P_{M^\perp}Z_d\big|_{M^\perp}\right)\,.
\ees
If $[A_i,A_j^*] \in \cL^p$ for all $i,j$ then $M$ is said to be $p$-essentially normal, and if $[A_i,A_j^*]$ is compact for all $i,j$ then $M$ is said to be essentially normal. Similarly, the quotient $\arv /M$ is said to be $p$-essentially normal (resp. essentially normal) if the commutators $[B_i,B_j^*]$ are all in $\cL^p$ (resp. compact). 

Arveson conjectured that every graded submodule $M$ of $\arv$, as well as its quotient $\arv /M$, are $p$-essentially normal for $p>d$ \cite{Arv05}. This has been verified for modules generated by monomials \cite{Arv05,Doug06}, and also for principal modules as well as arbitrary modules in dimensions $d=2,3$ \cite{GuoWang}. Douglas conjectured further that $\arv/M$ is $p$ essentially normal for all $p>\dim(M)$ \cite{Doug06b}. This has also been verified in several cases. We will not discuss here the varied and important consequences of this conjecture (see \cite{Arv05,Arv07,Doug06,Doug06b,GuoWang}).

In Section \ref{sec:stab_div_ess_nor} we will show that every module that has the stable division property satisfies Douglas' refinement of Arveson's conjecture. Thus, having the results of Sections \ref{sec:H^2} and \ref{sec:ell1} at hand, we obtain a unified proof that principal modules, monomial modules, and arbitrary ideals in $\mb{C}[x,y]$ are $p$-essentially normal for $p>d$, and that their quotients are $p$-essentially normal for $p> \dim(M)$.

\subsection*{Acknowledgments.} The author would like to thank Ken Davidson, J\"org Eschmeier and Chris Ramsey for reading and discussing preliminary versions of these notes. Moreover, the generous and warm hospitality provided by Ken Davidson at the University of Waterloo is greatly appreciated.

\section{Stable division with respect to the $H^2$ norm}\label{sec:H^2}

In this section $\|\cdot\|$ denotes the $H^2$ norm given by (\ref{eq:H2}), though the results here can be shown to be true also for other natural norms, in particular for the $\ell_1$ norm. The following is the simplest example.

\begin{proposition}\label{prop:orthogonal}
Let $I = I_1 + I_2 + \ldots$ be a homogeneous ideal in $\cA_d$ generated  by an orthonormal set $\{f_1, \ldots, f_k\}$ of linear polynomials. For every $n\geq 1$, every $g \in I_n$ can be written as $g = a_1 f_1 + \ldots + a_k f_k$, where $a_i \in H_{n-1}$ for $i=1, \ldots, k$ in such a way that $a_m f_m \perp  a_{j} f_{j}$ for all  $i \neq j$. In particular, $I$ has the stable division property.
\end{proposition}
\begin{proof}
We may assume that $f_i = z_i$, the first coordinate function, for $i=1,2,\ldots,k$. (see the corollary to Proposition 1.12 in \cite{Arv98}). Every polynomial $g$ is a sum of monomials of degree $n$. Take all monomials that have $z_1$ in them, and gather them up as $a_1 f_1$. All the remaining monomials in $g-a_1 f_1$ do not have $z_1$ in them, so they are orthogonal to $a_1f_1$. Proceeding inductively we are done. 
\end{proof}

We note that the conclusion in the above proposition does not hold if $\{f_1, \ldots, f_k\}$ is an orthonormal set of linear, vector valued polynomials in $H_d^2 \otimes \mb{C}^r$.

\subsection{Monomial modules}

A \emph{monomial} is a polynomial of the form $z^\alpha \otimes \xi$, with $\alpha$ a multi-index and $\xi \in  \mb{C}^r$ (note that this definition of monomial is more general than that given in \cite{CLO98}).

\begin{proposition}\label{prop:monomials}
Let $M \subset \cA_d \otimes \mb{C}^r$ be a module that is generated by monomials. Then $M$ has the stable division property. Moreover, the constant $C$ in (\ref{eq:stablediv2}) can be chosen to be $1$.
\end{proposition}
\begin{proof}
By Hilbert's Basis Theorem, there is some $m$ and a finite family $B = \{z^{\alpha_i} \otimes \xi_i\}_{i=1}^k \subseteq M_m$ 
that generates $M_m + M_{m+1} + \ldots $. 
A Graham-Schmidt orthogonalization procedure puts us in the situation where whenever $\alpha_i = \alpha_j$ then $\xi_i \perp \xi_j$. 
Throwing in finite orthonormal bases of $M_1, \ldots, M_{n-1}$ allows us to restrict attention to stable division in $M_m + M_{m+1} + \ldots$, so let us assume that $B$ generates $M$. Under these assumptions, we proceed by induction on $k$.

We have already noted that a principal submodule has the stable division property, so if $k = 1$ we are done. 

Now let $k>1$, and fix $h \in M_n$, $n \geq m$. $h$ can be written as a sum of monomials
\bes
h = \sum_{|\beta| = n} z^\beta \otimes \eta_\beta .
\ees
We re-label the set $\{z^{\alpha_i} \otimes \xi_i\ | \alpha_i = \alpha_1\}$ as $\{z^{\alpha_1} \otimes \zeta_j\}_{j=1}^t$. Remember that by our assumptions, $\{\zeta_1, \ldots, \zeta_t\}$ is an orthonormal set. Let $W = \textrm{span}\{\zeta_1, \ldots, \zeta_t\}$. Put 
$$S(\alpha_1) = \{\beta : |\beta| = n \, , \, \beta \geq \alpha_1 \}.$$
For all $\beta \in S(\alpha_1)$, $\eta_{\beta} = v_\beta + u_\beta$, with $v_\beta \in W$ and $u_\beta \in W^\perp$. Define
\bes
g = \sum_{\beta \in S(\alpha_1)} z^\beta \otimes v_\beta .
\ees
$g$ is in the module generated by $\{z^{\alpha_i} \otimes \xi_i\ | \alpha_i = \alpha_1\}$. Writing $v_\beta = \sum_{j=1}^t c^\beta_j \zeta_j$, we find that 
\bes
g = \sum_{j=1}^t \left( \sum_{\beta \in S(\alpha_1)} c^\beta_j z^{\beta-\alpha_1} \right) z^{\alpha_1} \otimes \zeta_j ,
\ees
so that gives $g = \sum_j a_j z^{\alpha_1} \otimes \zeta_j$ with $\sum_j \|a_j z^{\alpha_1}\otimes \zeta_j\|^2 \leq \|g\|^2$.
Now, $g \perp h-g$, and $h-g$ is in the module generated by $\{z^{\alpha_i} \otimes \xi_i | \alpha_i \neq \alpha_1\}$. By the inductive hypothesis, we can find a set of polynomials $\{b_i\}$ such that
\bes
h-g = \sum_{\alpha_i \neq \alpha_1} b_i z^{\alpha_i} \otimes \xi_i
\ees
and $\sum \|b_i z^{\alpha_i} \otimes \xi_i\|^2 \leq \|h-g\|^2$. Thus 
\bes
h = \sum_j a_j z^{\alpha_1} \otimes \zeta_j + \sum_{\alpha_i \neq \alpha_1} b_i z^{\alpha_i} \otimes \xi_i
\ees
with 
\bes
\sum\|a_j z^{\alpha_1} \otimes \zeta_j  \|^2 + \sum \| b_i z^{\alpha_i} \otimes \xi_i\|^2 \leq \|h\|^2.
\ees
\end{proof}

\subsection{Ideals in $\mb{C}[x,y]$}

We now consider the case of two variables, that is, $d=2$. 

\begin{lemma}\label{lem:stab2}
Let $f_1, \ldots, f_k$ be homogeneous polynomials of the same degree $m$ in $\mb{C}[x,y]$ such that $LT(f_1) > LT(f_2) > \ldots > LT(f_k)$. There is a constant $C$ such that for every polynomial $h \in \mb{C}[x, y]$, division of $h$ by $(f_1,\ldots,f_k)$ gives a representation  
\bes
h = a_1 f_1 + \ldots + a_k f_k + r,
\ees
with 
\be\label{eq:stabdivr}
\sum_i \|a_i f_i\|^2 \leq C (\|h\|^2 + \|r\|^2) ,
\ee
where $a_i, r \in \cA_d$, and either $r = 0$ or $r$ is a linear combination of monomials, non of which is divisible by any of $LT(f_1), \ldots, LT(f_k)$.
\end{lemma}
\begin{proof}
Note that we need only consider homogeneous $h$ - otherwise we apply the result to the homogeneous  components of $h$. We may also assume that $\deg h > 4m$.

We will use {\bf Algorithm I} from Appendix \ref{subsec:div_alg} for the division, where in step (\ref{it:choice}) we will choose $i_0 = \max I$. What remains to show will be proved by showing that the output of the algorithm described in Appendix \ref{subsec:div_alg} satisfies the required conditions, once the input is arranged so that $LT(f_1) > LT(f_2) > \ldots > LT(f_k)$, and as long $i_0$ is chosen as above.

The only change from the algorithm given in \cite[p. 63]{CLO92} is the specification of the $f_i$ that is used to reduce $p$ in step (\ref{it:reduce}). The correctness of this algorithm is proved in \cite{CLO92} and is independent of the choice of the dividing $f_i$ in step (\ref{it:reduce}). It remains to prove that there exists $C$ such that (\ref{eq:stabdivr}) holds.

The proof is by induction on $k$ - the number of the $f_i$'s given. If $k=1$ the result is trivial.
Assume that $k>1$. Write  $f_i = \sum_{j=0}^m a_{ij} x^{m-j}y^j$, and for all $i$, put $j_i = \min\{j | a_{ij} \neq 0\}$. By assumption, $j_1 < j_2 < \ldots < j_k$.

Recall that we may assume that $\deg h = n > 4m$. From the definition of the algorithm it follows that $f_1$ will be used in step (\ref{it:reduce}) to divide $p$ only when the leading term of $p$ is of the form $b_t x^{n-t}y^t$, with $b_t \neq 0$ and $j_1 \leq t < j_2$.
By the triangle inequality, at every iteration in which $a_1$ changes, the quantity $\|a_1 f_1\|$ grows by at most $\|LT(p)/LT(f_1) f_1\|$. 

\noindent{\bf Claim:} $\|LT(p)/LT(f_1) f_1\|^2 \leq |a_{1j_1}|^{-1} \|LT(p)\|^2 \sum_j |a_{1j}|^2 $.

\noindent{\bf Proof of Claim:} 
\begin{align*}
LT(p)/LT(f_1) f_1 &= \frac{b_t}{a_{1j_1}} x^{n-t-(m-j_1)}y^{t-j_1} \sum_j a_{1j} x^{m-j}y^j \\
&= \sum_j a_{1j}  \frac{b_t}{a_{1j_1}} x^{n-t-(j-j_1)}y^{t+j-j_1}.
\end{align*}
Thus, by the definition of the norm in $H^2_2$,
\bes
\|LT(p)/LT(f_1) f_1\|^2 =  \left|\frac{b_t}{a_{1j_1}}\right|^2 \sum_{j} |a_{1j}|^2 \frac{(n-(t+j-j_1))!(t+j-j_1)!}{n!}
\ees
But $t \leq (t+j-j_1) < n/2$, and for integers $i,j$ such that $i \leq j<n/2$ we have
\bes
\frac{(n-j)!j!}{n!} \leq \frac{(n-i)!i!}{n!},
\ees
so
\begin{align*}
\|LT(p)/LT(f_1) f_1\|^2 &\leq \left|\frac{b_t}{a_{1j_1}}\right|^2 \sum_j |a_{1j}|^2 \frac{(n-t))!t!}{n!} \\
&=  |a_{1j_1}|^{-1} \|LT(p)\|^2 \sum_j |a_{1j}|^2 .
\end{align*}
That establishes the claim.

Now, we have seen that at every step of the iteration where $a_1$ changes, the quantity $\|a_1 f_1\|$ grows by as most 
$(\sum_j |a_{1j}|^2 \|LT(p)\|^2)^{1/2}$. At every such iteration, $\|p\|$ also grows by at most $(\sum_j |a_{1j}|^2 \|LT(p)\|^2)^{1/2}$. At the iterations where $a_1$ does not change, $\|p\|$ becomes smaller. 

It follows that after at most $j_2$ iterations, we have the following situation: 
\begin{enumerate}
\item $\|a_1 f_1\| \leq C \|h\|$.
\item $\|p\| \leq C \|h\|$.
\item $r$ is something.
\item $a_2 = \ldots a_k = 0$.
\end{enumerate}
Here $C$ is a constant that depends only on $\sum_j |a_{1j}|^2$ and $j_2$. From this stage on, the algorithm continues to divide $p$ by $f_2, \ldots, f_k$. It will find the same $a_2, \ldots, a_k$ that it would given $p$ instead of $h$ as input, and it would add to $r$ a remainder that is orthogonal to the remainder $r$ that is obtained when we are done with $f_1$.
By the inductive hypothesis, 
\bes
\sum_{i=2}^k \|a_i f_i\|^2 \leq C' (\|p\|^2 + \|r\|^2) \leq C' (C\|h\| + \|r\|^2).
\ees
Putting this together with $\|a_1 f_1\| \leq C \|h\|$, and changing $C$, we are done.
\end{proof}

\begin{remark}
\emph{It would be desirable to replace (\ref{eq:stabdivr}) with the stronger $\sum_i \|a_i f_i\|^2 + \|r\|^2 \leq C' \|h\|^2$, but that is impossible. For example, when $k=1$ and $f_1 = x^2 + xy$, running the algorithm with the input $h = x^n$ will give  huge remainders $r$ (see Example \ref{ex:x^2+xy}).}
\end{remark}

\begin{theorem}\label{thm:d=2}
Every homogeneous ideal $I \subseteq \cA_2$ has the stable division property. 
\end{theorem}

\begin{proof}
As in the proof of Proposition \ref{prop:monomials}, we may assume that that $I$ is generated by a set $F = \{f_1, \ldots, f_k\}$ of homogeneous polynomials of the same degree $m$. Furthermore, we may assume that $F$ is a Groebner basis with respect to lexicographic order on monomials.

By Lemma \ref{lem:stab2}, there is a $C$ such that every $h \in \cA_2$ can be written as 
\bes
h = a_1 f_1 + \ldots + a_k f_k + r,
\ees
with $\sum_i \|a_i f_i\|^2 \leq C(\|h\|^2 + \|r\|^2)$. Now let $h \in I_n$. We may assume that $n > 4m$. Under this assumption, we saw that the $a_i$'s and $r$ can be found by the division algorithm. But by the Corollary on p. 81, \cite{CLO92}, since $F$ is a Groebner basis, we actually get $r = 0$. Thus
\bes
\sum_i \|a_i f_i\|^2 \leq C\|h\|^2
\ees
for all such $h$, and the proof is complete.
\end{proof}

The following example shows that Lemma \ref{lem:stab2} cannot be extended to $d>2$.
\begin{example}\label{expl:non_stable_basis}
\emph{Taking $f_1 = x^2+wy, f_2 = y^2$,  and $h = x^4 w^n$, we find that the above algorithm gives}
\bes
h = (x^2 w^n - w^{n+1}y) f_1 + w^{n+2}f_2.
\ees
\emph{But $\|h\|^2 \sim n^{-4}$, while $\|w^{n+2}f_2\|^2 = \|w^{n+2}y^2\|^2 \sim n^{-2}$. In fact, in any presentation of $h$ as a combination $h = a_1 f_1 + a_2 f_2$, the monomial $w^{n+2}y^2$ must appear in both terms $a_1 f_1$ and $a_2 f_2$. That means that we cannot write $h = a_1 f_1 + a_2 f_2$ with $\|a_1 f_1 \|^2 + \|a_2 f_2\|^2 \leq C \|h\|^2$, where $C$ is independent of $h$. So the the set of generators $\{x^2 + wy, y^2\}$ is not a stable generating set for the ideal $I = \lel x^2 + wy, y^2 \rir$ that it generates. It is worth noting that $\{x^2 + wy, y^2\}$ is a Groebner basis for I. On the other hand, the ideal $I$ \emph{does} have the stable division property. This can be verified by using a Groebner basis with respect to the lexicographic order with $w>x>y$. This Groebner basis is given by $\{y^2, yx^2, x^4, wy + x^2\}$. }
\end{example}

\subsection{Zero dimensional ideals}
Recall that an ideal $I \subseteq \cA_d$ is said to be \emph{zero dimensional} if the affine variety associated to $I$,
\bes
V(I):=\{z \in \mb{C}^d : \forall f \in I . f(z) = 0\},
\ees
is finite. Note that for a zero dimensional \emph{homogeneous} ideal $I$ it is always true that $V(I) = \{0\}$.

\begin{theorem}
Let $I$ be any zero dimensional ideal in $\cA_d$. Then $I$ has the stable division property.
\end{theorem}
\begin{proof}
By the theorem on page 232 in \cite{CLO92}, $I$ is a finite co-dimensional subspace of $\cA_d$, and from here it is not hard to prove that it has the stable division property.
\end{proof}

\section{Stable division with respect to the $\ell_1$ norm}\label{sec:ell1}

In this section $\|\cdot\|$ denotes the $\ell_1$ norm given by (\ref{eq:l1}). This norm is perhaps the most natural way to measure the ``size" of a polynomial, and it also has the feature that it behaves nicely with respect to the division algorithm (roughly speaking, the division algorithm moves coefficients from one coordinate to another, therefore an $\ell_1$ norm is more appropriate than an $\ell_2$ norm). All the classes of modules that were shown in the previous section to have the stable division property with respect to the $H^2$ norm can also be seen (using the same proofs) to have the stable division property with respect to the $\ell_1$ norm. However, for the $\ell_1$ norm we can prove much more. We shall show in this section that every ideal is \emph{linearly equivalent} to an ideal that has the stable division property (see Definition \ref{def:lineq} below).

In this section, unlike the rest of the paper, it will be convenient to use the lexicographic order with $z_d > \ldots > z_1$.

A straightforward calculation gives the following lemma.
\begin{lemma}\label{lem:M_f}
Let $f \in \cA_d$, and let $M_f: \cA_d \rightarrow \cA_d$ be the operator given by
\bes
M_f g = fg.
\ees 
Then $\|M_f\| = \|f\|$.
\end{lemma}

\begin{proposition}\label{prop:stable_div_cond}
Let $f_1, \ldots, f_k \in \cA_d$ be such that for all $j=1, \ldots,k$, if $f_j(z) = \sum_\alpha c_\alpha z^\alpha$ with $LT(f_j) = c_\beta z^\beta$, then
\be\label{eq:cineq}
|c_\beta| >  \sum_{\alpha \neq \beta} |c_\alpha|.
\ee
Then there is a constant $C$ such that for every $h \in \cA_d$, the division algorithm gives a decomposition
\be\label{eq:decomposition}
h = \sum_{i=1}^k a_i f_i + r,
\ee
with $\sum_i \|a_i f_i \| \leq C \|h\|$ and $\|r\| \leq h$.
\end{proposition}
\begin{proof}
It convenient to assume that the leading coefficients of the $f_j$'s are all $1$, and we may do so. Thus there is some $\rho\in(0,1)$ such that for all $j$, if $f_j(z) = \sum_\alpha c_\alpha z^\alpha$ with $LT(f_j) = c_\beta z^\beta$, then $\sum_{\alpha \neq \beta} |c_\alpha| < \rho$.

Let $h(z) = \sum b_\alpha z^\alpha$. We run {\bf Algorithm II} from Appendix \ref{subsec:div_alg} (please recall the notation). Using condition (\ref{eq:cineq}), it is easy to see that every modification of $p$ in Step (\ref{it:reduceII}), $\|p\|$ only gets smaller. Since at the beginning of the algorithm we set $p:=h$, and at the end of the algorithm we set $r:=p$, we get $\|r\| \leq \|h\|$.

Now we must also bound the quantity $\sum \|a_i f_i\|$. By Lemma \ref{lem:M_f}, it is enough to to bound $\sum_i \|a_i\|$ by a multiple of $\|h\|$. The rest of the proof is devoted to obtaining the bound
\be\label{eq:bound}
\sum_{i=1}^k \|a_i\| \leq (1-\rho)^{-1} \|h\| .
\ee
We introduce some notation to streamline the slightly technical argument. For a monomial term $h = cz^\gamma$ ($c\neq 0$), we define the \emph{hight} of $h$, denoted $Ht(h)$, as
\bes
Ht(h) := |\{\beta : \beta \leq \gamma \}|,
\ees
where $|\cdot|$ denotes cardinality. For a general polynomial $h$ we define $Ht(h) = Ht(LT(h))$.

To algorithmically obtain (\ref{eq:decomposition}) with estimate (\ref{eq:bound}), we need to specify the choice of term made in Step (\ref{it:ChooseTerm}) in {\bf Algorithm II}. The specifications needed will be made clear by the proof below. The reader may later want to check that the procedure implied by the proof below is equivalent to choosing at each iteration of Step (\ref{it:ChooseTerm}) the term $t$ of $p$ that is the \emph{minimal} possible term reducible by any $f_j$. 
We will prove (\ref{eq:bound}) by induction on the height of $h$.

\noindent{\bf Claim:} \emph{Division of a polynomial $h$ by $(f_1, \ldots, f_k)$ gives the decomposition (\ref{eq:decomposition}) such that}
\be\label{eq:Htbound}
\sum_{i=1}^k\|a_i\| \leq \sum_{n=0}^{Ht(h)} \rho^n \|h\|.
\ee

\noindent{\bf Proof of claim.} If $Ht(h) = 1$ then $h$ is a nonzero constant. Either it plays the role of the remainder in (\ref{eq:decomposition}), or one of the $f_i$'s is a constant, say $f_1 = c$, and  then $h = h/c f_1$. In this case (\ref{eq:Htbound}) trivially holds.

Assume now that $Ht(h) > 1$. Write $h = cz^\gamma + g$, where $cz^\gamma  = LT(h)$ and $g = h - LT(h)$. Note that $\|h\| = \|cz^\gamma\| + \|g\|$. Algorithmically, we will first divide $g$ and only then shall we turn to dividing $cz^\gamma$. This is equivalent to dividing $g$ and $cz^\gamma$ separately and then adding the output. Since $Ht(g) < Ht(h)$, the inductive hypothesis gives
\bes
g = \sum_{i=1}^k a_i^1 f_i + r^1
\ees
with $\sum\|a_i^1\| \leq  \sum_{n=0}^{Ht(h)-1} \rho^n \|g\|$. Now we consider the term $cz^\gamma$. If it is not divisible by any of the leading terms of $f_1, \ldots, f_k$ then we have equation (\ref{eq:decomposition}) with $a_i = a_i^1$ and $r = r^1 + cz^\gamma$. In this case the required bound holds. 

If $cz^\gamma$ is divisible by one of the leading terms of $f_1, \ldots, f_k$, say by $LT(f_{i_0})$, then we reduce the term $t = cz^\gamma$ by $f_{i_0}$ as described in Step \ref{it:reduceII} of {\bf Algorithm II}: $A_{i_0} := cz^\gamma /LT(f_{i_0})$ and $p := cz^\gamma - (cz^\gamma /LT(f_{i_0}))f_{i_0}$. This step produces a polynomial $p$ which we need to continue to divide. Note that $\|p\| \leq \rho \|cz^\gamma\|$ and $Ht(p) < Ht(cz^\gamma)$. By the inductive hypothesis, division of $p$ gives
\bes
p = \sum_{i=1}^k a_i^2 f_i + r^2 ,
\ees
with $\sum\|a^2_i\| \leq \sum_{n=0}^{Ht(h)-1} \rho^n  \|p\|$. Thus we have equation (\ref{eq:decomposition}) with $a_i = a_i^1 + a_i^2$ for $i \neq i_0$, $a_{i_0} = a^1_{i_0} + a^2_{i_0} + A_{i_0}$, and $r = r^1 + r^2$.  Thus 
\begin{align*}
\sum \|a_i\| &\leq \sum \|a_i^1\| + \sum \|a_i^2\| + \|A_{i_0}\| \\
&\leq \sum_{n=0}^{Ht(h)-1} \rho^n \|g\| + \sum_{n=0}^{Ht(h)-1} \rho^n  \|p\| + \|cz^\gamma\| \\
&\leq \sum_{n=0}^{Ht(h)-1} \rho^n \|g\| + \sum_{n=1}^{Ht(h)} \rho^n  \|cz^\gamma\| + \|cz^\gamma\| \\
&\leq \sum_{n=0}^{Ht(h)} \rho^n  (\|g\|+\|cz^\gamma\|) = \sum_{n=0}^{Ht(h)} \rho^n  \|h\|.
\end{align*}
That proves the claim, which clearly implies (\ref{eq:bound}). As we noted earlier, this bound together with Lemma \ref{lem:M_f} completes the proof.
\end{proof}

\begin{definition}\label{def:lineq}
We say that two ideals $I,J \subseteq \cA_d$ are \emph{linearly equivalent} if there is a linear change of variables that sends $I$ onto $J$.
\end{definition}

\begin{lemma}\label{lem:existlambda}
Let $f_1, \ldots, f_k \in \cA_d$. There exist $\lambda_1, \ldots, \lambda_d > 0$ such that the polynomials $g_1, \ldots, g_k$ given by 
\bes
g_j(z_1, \ldots, z_d) = f_j(\lambda_1 z_1, \ldots, \lambda_d z_d)
\ees
satisfy the following: for all $j=1, \ldots,k$, if $g_j(z) = \sum_\alpha c_\alpha z^\alpha$ with $LT(g_j) = c_\beta z^\beta$, then
\bes
|c_\beta| > \sum_{\alpha \neq \beta} |c_\alpha|.
\ees
\end{lemma}
\begin{proof}
We may assume that not all the $f_j$'s are monomials. Put $N = \max_j \deg f_j$, and let $M$ be the dimension of the space of polynomials with degree less than or equal to $N$. Define $K = \max\{|c| : c \textrm{ is a coefficient of some } f_j\}$. Define $\lambda_1 = M(K+1)$, and now define $\lambda_2, \ldots, \lambda_d$ recursively by
\bes
\lambda_{j+1} = (\lambda_1 \cdots \lambda_j)^{N+1} \,\, , \,\, j = 1, \ldots, d-1.
\ees
Now let $j\in \{1, \ldots,k\}$, and consider $g_j(z) = \sum_\alpha c_\alpha z^\alpha$. The choice of the $\lambda_i$'s implies that whenever $c_\alpha z^\alpha < c_\beta z^\beta$, then $M (K+1) |c_\alpha| < |c_\beta|$. The result follows.
\end{proof}

\begin{lemma}\label{lem:GrobnerBasis}
Let $I$ be an ideal, let $\{f_1, \ldots, f_k\}$ be a Groebner basis for $I$, and fix $\lambda_1, \ldots, \lambda_d \in \mb{C} \setminus \{0\}$. Define
\be\label{eq:J}
J = \{f(\lambda_1 z_1, \ldots, \lambda_d z_d) : f \in I\},
\ee
and 
\be\label{eq:g_j}
g_j(z_1, \ldots, z_d) = f(\lambda_1 z_1, \ldots, \lambda_d z_d) \,\, , \,\, j = 1, \ldots, k.
\ee
Then $J$ is an ideal that is equivalent to $I$ for which $\{g_1, \ldots, g_k\}$ is a Groebner basis.
\end{lemma}
\begin{proof}
Note that $LT(J) = LT(I)$, and that for all $j$, up to multiplication by constants, $LT(f_j) = LT(g_j)$. Thus $\lel LT(g_1), \ldots, LT(g_k) \rir = LT(J)$, thus $\{g_1, \ldots, g_k\}$ is a Groebner basis for $J$.
\end{proof}

\begin{theorem}
Every ideal in $\cA_d$ is linearly equivalent to an ideal that has the stable division property with respect to the $\ell_1$ norm.
\end{theorem}
\begin{proof}
Let $I$ be an ideal in $\cA_d$. Let $\{f_1, \ldots, f_k\}$ be a Groebner basis for $I$. Define $J$ as in (\ref{eq:J}), and define $\{g_1, \ldots, g_k\}$ as in (\ref{eq:g_j}). By Lemma \ref{lem:GrobnerBasis}, $\{g_1, \ldots, g_k\}$ is a Groebner basis for $J$, and $J$ is equivalent to $I$, for any choice of nonzero $\lambda_1, \ldots, \lambda_k$. By Lemma \ref{lem:existlambda}, we can find such $\lambda$'s for which $g_1, \ldots, g_k$ satisfy the condition of Proposition \ref{prop:stable_div_cond}. But every $h \in J$ is divisible by $\{g_1, \ldots, g_k\}$ with remainder zero, so Proposition \ref{prop:stable_div_cond} implies that $J$ has the stable division property.
\end{proof}

This theorem shows that the stable division property, at least with respect to the $\ell^1$ norm, has nothing to do with the \emph{geometry} of an ideal (in the sense of algebraic geometry). That is: either all ideals have the stable division property, or there exists an ideal that does not posses this property, but which is equivalent to one that does.

\section{Stable division and essential normality}\label{sec:stab_div_ess_nor}

Let $M$ be a graded submodule of $H^2_d \otimes \mb{C}^r$. It is known that there exists a univariate polynomial $HP_M(t)$ such that $\dim (M_n^\perp) = HP_M(n)$ for $n$ sufficiently large \cite[Proposition 4.7]{CLO98}. We define the dimension of $M$, denoted $\dim(M)$, to be  $\deg (HP_M) + 1$. When $r=1$ and $M$ is an ideal, then $\dim(M)$ is the dimension of the \emph{affine} variety determined by $M$.
Since $\dim(H_n \otimes \mb{C}^r) \sim c n^{d-1}$, we always have that $\dim(M) \leq d$.

\begin{theorem}\label{thm:stabdivessnorm}
Let $M$ be a graded Hilbert submodule of $H^2_d \otimes \mb{C}^r$ that has the stable division property. Then $M$ and $\arv /M$ are $p$-essentially normal for all $p>d$. In fact, $\arv /M$ is $p$-essentially normal for all $p> \dim(M)$.
\end{theorem}
 
\begin{proof}
It suffices to prove the assertion for $\arv/M$ \cite[Proposition 4.2]{Arv07}. $\arv/M$ is unitarily equivalent, as a Hilbert module, to $M^\perp$, where the coordinate functions are given by compressing $Z_1, \ldots, Z_d$ to $M^\perp$.  

Let $P$ be the orthogonal projection onto $M^\perp$. Denote $B_i = P Z_i \big|_{M^\perp}$. Fix $i,j$ and $p>\dim(M)$. What we need to prove is that
\bes
[B_i,B_j^*] = B_iB_j^* - B_j^*B_i \in \cL^p .
\ees
We know that $\|[Z_i,Z_j^*]  \big|_{H_n} \| \leq \frac{2}{n+1}$ \cite[Proposition 5.3]{Arv98}, therefore
\bes
\textrm{trace}(|P[Z_i,Z_j^*]P|^p) \leq \sum_n \frac{2 \dim(M_n^\perp)}{(n+1)^p} < \infty .
\ees
Thus it is equivalent to show that $[B_i,B_j^*] - P[Z_i,Z_j^*]P$ is in $\cL^p$. But
\bes
[B_i,B_j^*] - P[Z_i,Z_j^*]P = PZ_iPZ_j^*P - PZ_j^*PZ_iP - PZ_iZ_j^*P + PZ_j^*Z_i P  = PZ_j^* (I - P) Z_i P ,
\ees 
where we used $PZ_j^*P = Z_j^*P$ ($M^\perp$ is coinvariant). Letting $E_n$ denote the orthogonal projection $E_n : \arv \rightarrow H_n \otimes \mb{C}$, and putting $P_n = E_nP$, then we may write
\bes
PZ_j^* (I - P) Z_i P = \sum_n P_n Z_j^* (E_{n+1}-P_{n+1})Z_i P_n. 
\ees

The proof will be complete once we show that 
\be\label{eq:asnormestimate}
\|P_n Z_j^* (E_{n+1}-P_{n+1})\| \leq C(n+1)^{-1/2} ,
\ee
with $C$ independent of $n$. Indeed, this would imply that 
\bes
\textrm{trace}(|PZ_j^* (I - P) Z_i P|^p) \leq C' \sum_n \frac{n^{\dim(M)-1}}{(n+1)^{-p}} 
\ees 
(here, $C'$ is some other constant) which is finite for $p>\dim(M)$.

Let $F = \{f_1, \ldots, f_k\}$ be a stable generating set for $M$. Let $m$ be the maximal degree of an element in $F$. Modifying $F$ if needed, we may assume that $F \subset M_m$ is a stable generating set for $M_m + M_{m+1} + \ldots$. 

Now consider $n \geq m$, and let $h \in M_{n+1}$. Because $F$ is a stable generating set, we write 
$h = a_1 f_1 + \ldots + a_k f_k$, with $\sum_i \|a_i f_i\| \leq C\|h\|$.
Recalling that $Z^*_j\big|_{H_{n+1}} = (n+1)^{-1} \frac{\partial}{\partial z_j}$, we get
\bes
Z^*_j h = \frac{1}{n+1}\Big( \sum_{i=1}^k a_i  \frac{\partial}{\partial z_j} f_i + \sum_{i=1}^k f_i \frac{\partial}{\partial z_j} a_i \Big) ,
\ees
so, because $M$ is a submodule, 
\bes
P_n Z_j^* h = \frac{1}{n+1}\sum_{i=1}^k a_i  \frac{\partial}{\partial z_j} f_i  .
\ees
By \cite[Proposition 2.3]{GuoWang} there is a constant $C_1$ such that $\|g \partial/\partial z_j f_i\| \leq C_1\sqrt{n+1}\|gf_i\|$ for $i=1, \ldots, k$, and we get
\begin{align*}
\|P_n Z_j^* h\| &\leq  \frac{1}{n+1}\sum_{m=1}^k \|a_m  \frac{\partial}{\partial z_j} f_m\| \\
&\leq  \frac{C_1 \sqrt{n+1}}{n+1}\sum_{m=1}^k \|a_m  f_m\| \\
&\leq  \frac{C C_1 \|h\|}{\sqrt{n+1}}.
\end{align*}
That establishes (\ref{eq:asnormestimate}), and completes the proof of the theorem.
\end{proof}

Using the theorem together with the results of Section \ref{sec:H^2}, we obtain a unified proof for the following known results:

\begin{theorem}[Guo-Wang \cite{GuoWang}]
Every principal graded submodule $M \subseteq \arv$, as well as its quotient, is $p$-essentially normal for all $p>d$. $\arv/M$ is $p$-essentially normal for $p>\dim(M)$.
\end{theorem}

\begin{theorem}[Guo-Wang \cite{GuoWang}]
Every homogeneous ideal $I$ in $H^2_2$, as well as its quotient, are $p$-essentially normal for $p>2$.
\end{theorem}

\begin{theorem}[Arveson \cite{Arv05}, Douglas \cite{Doug06}]
Let $f_1, \ldots, f_k$ be homogeneous vector valued polynomials of the same degree $m$, all of which are monomials. Then the module $M$ generated by $\{f_1, \ldots, f_k\}$, as well as its quotient, are essentially $p$-normal for all $p>d$. $\arv/M$ is $p$-essentially normal for $p>\dim(M)$.
\end{theorem}

\begin{remark}
\emph{In a previous version of this note it was only asserted that $\arv/M$ is $p$-essentially normal for $p>d$, rather than for $p>\dim(M)$. It was noticed that the proof gives the stronger result thanks to a correspondence with J\"org Eschmeier.}
\end{remark}

\section{Reduction from linear submodules of $H^2_d \otimes \mb{C}^r$ to quadratic submodules of $H^2_d$.}

The purpose of this section is to show that the problem of showing the $p+r$-essential normality of linear submodule of $\arv$ can be reduced to the problem of showing $p$-essential normality of quadratic submodules of $H^2_d$. 
The motivation for this reduction is, of course, Arveson's result that if every homogeneous submodule $M$ of $\arv$ that is generated by linear polynomials is essentially normal, then every graded submodule of $\arv$ (as well as its quotient) is essentially normal \cite[Corollary 8.4]{Arv07}\footnote{We note that it appears that the same proof given in \cite{Arv07} would give the same result for $p$-essential normality.}.

\noindent{\bf Statement:} \emph{If it is true that every homogeneous ideal in $H^2_d$ that is generated by quadratic polynomials is $p$-essentially normal for $p>d$, then every homogeneous submodule of $\arv$ that is generated by linear polynomials is $p$-essentially normal for all $p> d+r$. Similarly, if it is true that every homogeneous ideal in $H^2_d$ that is generated by quadratic polynomials is essentially normal, then every homogeneous submodule of $\arv$ that is generated by linear polynomials is essentially normal.}

\begin{proof}
We prove the statement about $p$-essential normality. The statement about essential normality is proved in a similar way. Fix $p>d+r$.
Write the $d$-dimensional variable as $z = (z_1, \ldots, z_d)$, and denote the coordinate operators by 
$S_1, \ldots, S_d$. Put $T_i = S_i \big|_{M}$, $i=1, \ldots, d$.

Let $M\subseteq \arv$ be generated by polynomials of degree $1$. Let $\{v_1, \ldots, v_r\}$ denote an orthonormal basis in $\mb{C}^r$. Let the generators $\{f_1, \ldots, f_k\}$ of $M_1$ be given by 
\bes
f_m(z) = \sum_{i,j}a^{m}_{ij} z_i v_j.
\ees

Now, consider the space $H^2_{d+r}$, with the $(d+r)$-dimensional variable written as $(z,y) = (z_1,\ldots, z_d, y_1, \ldots, y_r)$. We denote the coordinate operators of $H^d_{d+r}$ by $Z_1, \ldots, Z_d, Y_1, \ldots, Y_r$. Note that there is a difference between the tuples $(S_1, \ldots, S_d)$ and $(Z_1, \ldots, Z_d)$ - they are acting on different spaces and in a different way.
Define $k$ quadratic forms $g_1, \ldots, g_k$ by
\bes
g_m(z,y) = \sum_{i,j}a^{m}_{ij} z_i y_j.
\ees
Let $N$ be the graded Hilbert submodule of $H^2_{d+r}$ generated by $\{g_1, \ldots, g_k\}$. By assumption, $N$ is $p$-essentially normal.  In particular, letting $A_i = Z_i \big|_{N}$, we have that
\bes
A_i A_j^* - A_j^* A_i \in \cL^p \,\, , \,\, i,j = 1, \ldots, d .
\ees

Now, let $\cA$ be $\mb{C}[z_1, \ldots, z_d]$, considered as the subalgebra of $\mb{C}[z_1, \ldots, z_d, y_1, \ldots, y_r]$ consisting of polynomials depending only on the $z_i's$. $N$ is also an $\cA$-module. Let $P$ be the completion of the $\cA$-submodule of $N$ generated by $\{g_1, \ldots, g_k\}$. Denote $B_i = A_i\big|_{P}$.

With all these definitions set up, the proof will now be completed in two steps. First, we will show that for all $i=1, \ldots, d$, $P$ reduces $A_i$. As this obviously implies that $[B_i,B_j^*]$ are also in $\cL^p$, the second and final step will be to show that $p$-essential normality of $[B_i,B_j^*]$ implies $p$-essential normality of $[T_i,T_j^*]$.

\noindent {\bf 1. $P$ reduces $A_i$:}

$P$ is invariant for $A_i$ by definition. We need to show that $N \ominus P$ is also invariant under $A_i$. But $P$ consists of all polynomials in $N$ in which the $y$ variables appear in any term with degree precisely one. Thus $N \ominus P$ certainly contains the space of all polynomials in which the $y$ variables appear with degree strictly greater then $1$. Call this space $Q$. But $P + Q = N$, hence $N \ominus P = Q$. The definition of $Q$ as the space of polynomials in which the $y$ variables appear with degree strictly greater then $y$ implies that it is invariant under multiplication by $z_i$, i.e., it is invariant under  the operator $A_i$.

\noindent{\bf 2. $p$-essential normality of $[B_i,B_j^*]$ implies $p$-essential normality of $[T_i,T_j^*]$:} 

Let $R$ be the completion of the $\cA$-submodule of $H^2_{d+r}$ generated by $\{y_1, \ldots, y_r\}$. $R$ can be equivalently defined as
\bes
R = \{f\in H^2_{d+r} : \forall z,y,\lambda . f(z,\lambda y) = \lambda f(z,y)\}.
\ees
Define $U: H^2_d \otimes \mb{C}^r \rightarrow R$ on monomials by
\bes
U (z^{\alpha} v_j) = \sqrt{1+|\alpha|} z^\alpha y_j.
\ees
Using the formula
\bes
\|z^\alpha \|^2 = \frac{\alpha_1 ! \cdots \alpha_d !}{|\alpha|!},
\ees
one sees that $U$ extends to a unitary. From our definitions it follows that $U$ maps $M$ onto $P$. A simple computation shows:
\be\label{eq:almostunitary}
U^* Z_i U (z^{\alpha} v_j) = \sqrt{\frac{|\alpha|+1}{|\alpha| +2 }} S_i (z^\alpha v_j).
\ee
Let $D$ be the graded operator of degree $0$ on $H^2_d \otimes \mb{C}^r$, acting on the space of homogeneous polynomials of degree $n$ as multiplication by $\sqrt{n+1}/\sqrt{n}$. Then we can rewrite (\ref{eq:almostunitary}) as
\bes
D U^* B_i U =  T_i .
\ees

Further computations show that 
\bes
DU^* B_i U= D' U^* B_i U D  ,
\ees
where $D'$ is the operator that multiplies homogeneous polynomials of degree $n \geq 2$ by $\sqrt{(n-1)(n+1)}/n$. Now,
\begin{align*}
T_i T_j^* - T_j^* T_i &= D U^*B_i U  U^* B_j^* U D -  U^* B_j^* U D D U^*B_i U \\
&= D U^* B_i B_j^* U D -  DU^* B_j^* U^* D'^2 U B_i U D \\
&= D U^* B_i B_j^* U D -  DU^* B_j^* B_i U D    + DU^*  B_j^* U^*(I-D'^2)U B_i U D.
\end{align*}
Now, $D U^* B_i B_j^* U D -  DU^* B_j^* B_i U D  = D U^* [B_i, B_j^*] U D \in \cL^p$. On the other hand, $I - D'^2$ is the operator that multiplies the homogeneous polynomials of degree $n$ by $1 - (n-1)(n+1)/{n^2} = 1/{n^2}$, and it is not hard to see that this operator is in $\cL^q$ for all $q>d/2$. But $p>d+r$, so $DU^*  B_j^*U^*(I-D'^2)U B_i U D \in \cL^p$, and we are done.
\end{proof}

\section{Concluding remarks}

The problem of determining whether every homogeneous ideal in $\cA_d$ has the stable division property remains open. Besides being a compelling problem in its own right, and in addition to being directly related to questions of numerical stability in computational algebraic geometry, the consequence to essential normality of Hilbert modules serves as a great motivation for solving this problem. By the result of the previous section, it is already interesting to solve this problem for ideals generated by quadratic forms. But it is possible that even this problem is too hard to solve. 

The notion of stable division can be weakened  in several ways. One of these ways is to allow for \emph{approximate stable division}. That is, instead of requiring 
\bes
\sum_{i=1}^k a_i f_i = h
\ees
with $\sum \|a_i f_i\| \leq C \|h\|$, one requires only 
\bes
\|\sum_{i=1}^k a_i f_i - h \| \leq cn^{-1/2} \|h\|
\ees
and $\sum \|a_i f_i\| \leq C \|h\|$, where $n$ is the degree of $h$. It is then easy to see that, under the assumption of \emph{approximate} stable division, the proof of Theorem \ref{thm:stabdivessnorm} goes through. One can also allow for $C$ to be a slowly growing function of $n:= \deg h$, although that may affect the interval of those $p$ for which $p$-essential normality is shown (for example $C  \leq cn^{1/2-\epsilon}$ would still give $p$-essential normality for sufficiently large $p$, $C \leq c\log n$ would give $p$-essential normality for the same $p$'s, etc.). In fact, an analysis of the proof of Theorem \ref{thm:stabdivessnorm} shows that we may also allow the generating set to vary and in fact to having (slowly) growing degree.
These weakened notions of stable division are perhaps what one might hope to prove in order to establish Arveson's conjecture in general. 

Recently, J\"org Eschmeier developed a different approach to the problem of essential normality \cite{E10}. His approach is related to ours, but somewhat different in spirit. He showed that if an ideal $I$ is generated by homogeneous polynomials $\{f_1, \ldots, f_k\}$ of degree $m$, such that 
\be\label{eq:Eschmeier}
\|P_{I^\perp}\sum_{i=1}^k a_i  \frac{\partial}{\partial z_j} f_i  \| \leq C\sqrt{n} \|\sum_{i=1}^k a_i  f_i  \| 
\ee
holds for all $a_1, \ldots, a_d \in H_{n-m}$, then $H^2_d/I$ is $p$-essentially normal for all $p>\dim(I)$. He also showed that if $\{f_1, \ldots, f_k\}$ is a stable generating set for $I$, then $I$ has the above property.

\appendix
\section{The division algorithm and Groebner bases}\label{sec:Grobner}

We will use the notation of \cite{CLO92}, which is our main reference for the material reviewed in this section (see also \cite{CLO98,BW}). 

\subsection{Monomial orders}
The monomial order that we will use is the graded lexicographic order, which agrees with the usual lexicographic order on the space of homogeneous polynomials of a certain degree. Unless stated otherwise, we set that $z_1 > z_2 > \ldots > z_d$.
That is, for any $a,b \neq 0$, $a z^\alpha > b z^\beta$ if either $|\alpha| > |\beta|$, or $|\alpha| = |\beta|$ and the 
first non-zero entry in $\alpha - \beta$ is positive.

Given $p = \sum_\alpha c_\alpha z^\alpha \in \cA_d$, the \emph{leading term} of $p$, denoted $LT(p)$, is the monomial $c_\beta z^\beta$ appearing in $p$ that satisfies $c_\beta z^\beta > c_\alpha z^\alpha$ for all $\alpha \neq \beta$ such that $c_\alpha \neq 0$. 

\subsection{The division algorithm}\label{subsec:div_alg}
We now review the standard division algorithm given in \cite[p. 63]{CLO92} (which is identical to its finite multiplicity counterpart in \cite[p. 202]{CLO98}). In certain cases, once we carefully choose the order in which division is carried, we can prove that this algorithm implements stable division. 

Given polynomials $f_1, \ldots, f_k$ and another polynomial $h$, the purpose of this algorithm is to 
divide $h$ by $f_1, \ldots, f_k$ with remainder,  i.e., to exhibit $h$ as 
\bes
h = \sum_i a_i f_i + r,
\ees
where $a_1, \ldots, a_k$ are polynomials and $r$ is a polynomial that is to be considered as the ``remainder".

\noindent {\bf Algorithm I}

Given and ordered $k$-tuple $(f_1, \ldots, f_k)$ of polynomials and a polynomial $h$, set $a_1= \ldots = a_k = r = 0$, and set $p = h$.
While $p \neq 0$, execute the following steps:
\begin{enumerate}
\item If $p = 0$, then terminate and return the current values of $a_1, \ldots, a_k$ and $r$.
\item Set $I:= \{i | LT(f_i) \textrm{ divides }LT(p)\}$.
\item If $I = \emptyset$, put $r := r + LT(p)$, $p:= p - LT(p)$, and return to step (1); otherwise:
\item\label{it:choice} Choose by some method $i_0 \in  I$.  
\item\label{it:reduce} Put $a_{i_0} := a_{i_0} + LT(p)/LT(f_{i_0})$ and $p := p - (LT(p)/LT(f_{i_0}))f_{i_0}$, and return to step (1).
\end{enumerate}

Note that at the end of every iteration of the algorithm
\bes
h = \sum_{i=1}^k a_i f_i + p + r.
\ees

Here is a different version of the division algorithm, that is a little less intuitive but a little more flexible than the above one \cite[p. 199]{BW}.

\noindent {\bf Algorithm II}

Given and ordered $k$-tuple $(f_1, \ldots, f_k)$ of polynomials and a polynomial $h$, set $a_1= \ldots = a_k = 0$ and set $p = h$.
While there is a term $t$ in $p$ such that $p$ is divisible by one of the $LT(f_i)$'s, execute the following steps:
\begin{enumerate}
\item\label{it:ChooseTerm} Let $t$ be any term of $p$, chosen by some method, such that one of the $LT(f_i)$'s divides $t$.
\item Set $I:= \{i | LT(f_i) \textrm{ divides } t\}$.
\item\label{it:choiceII} Choose by some method $i_0 \in  I$.  
\item\label{it:reduceII} Put $a_{i_0} := a_{i_0} + t/LT(f_{i_0})$ and $p := p - (t/LT(f_{i_0}))f_{i_0}$, and return to step (1).
\end{enumerate}
When there are no more terms in $p$ that are divisible by any of the $LT(f_i)$'s, terminate and return the $a_i$'s and $r:=p$.

When organized this way, at every iteration of the algorithm we have
\bes
h =  \sum_{i=1}^k a_i f_i + p.
\ees

\subsection{Groebner bases}

Given an ideal $I \subseteq \cA_d$, a set $\{f_1, \ldots, f_k\}$ is called a \emph{basis} for $I$ if the ideal generated by $\{f_1, \ldots, f_k\}$ is $I$. The set $LT(I) = \{LT(p) : p \in I\}$ is always an ideal, and the set $\{f_1, \ldots, f_k\}$ is said to be a \emph{Groebner basis} for $I$ if the ideal generated by $\{LT(f_1), \ldots, LT(f_k)\}$ is $LT(I)$. It is a fact that every ideal has a Groebner basis, and that a Groebner basis is a basis. Moreover, if $\{f_1, \ldots, f_k\}$ is a Groebner basis for $I$, then when either one of the division algorithms is run with $h \in I$ and $(f_1, \ldots, f_k)$ as the dividing $k$-tuple, then the remainder is zero.


\bibliographystyle{amsplain}

\end{document}